\documentclass[letter, 12pt]{amsart}
\usepackage[dvipsnames,usenames]{color}
\newcommand{\floor}[1] {\left\lfloor #1 \right\rfloor}

\usepackage{extsizes}
\usepackage{blindtext}
\usepackage[utf8]{inputenc}
\usepackage{amsmath}
\usepackage{amssymb}

\usepackage{amsthm}
\usepackage{pst-node}
\usepackage{blindtext}
\usepackage{geometry}
 \usepackage{float}
 \usepackage{amssymb}
\usepackage{tikz-cd}
\usepackage{amscd}
\usepackage[all,color]{xy}
\usepackage{sseq}
\usepackage{amsfonts}
\usepackage{enumerate}
\usepackage{graphicx}
\usepackage{fancyhdr}
\usepackage{tikz}
\usetikzlibrary{cd}
\usetikzlibrary{arrows}
\usetikzlibrary{decorations.pathreplacing}
\usepackage{hyperref}
\usepackage{mathtools}
\usepackage{bm}
\usepackage{subfig}
\theoremstyle{plain} \numberwithin{equation}{section}

\newtheorem{theorem}{Theorem}[section]

\newtheorem{corollary}[theorem]{Corollary}

\newtheorem{lemma}[theorem]{Lemma}

\newtheorem{proposition}[theorem]{Proposition}
\theoremstyle{definition}
\newtheorem{definition}[theorem]{Definition}

\newtheorem{remark}[theorem]{Remark}

\newcommand{\x}{\times}

\newcommand{\s}{$\mathfrak{s}$}
\def \f-{f^{-1}}
\def \fp-{f^{-1}_\partial}

\usepackage{hyperref}

\def\Z{\mathbb{Z}}

\def\F{\mathbb{F}}

\def\II {\mathcal{I}}

\def\JJ {\mathcal{J}}

\def\l {\ell}

\def\x {\mathbf{x}}
\DeclareMathOperator{\gr}{gr}

\def\CF {\operatorname{CF}}
\def\HF {\operatorname{HF}}

\def\CFp {\operatorname{CF}^+}
\def\CFm {\operatorname{CF}^-}
\def\CFi {\operatorname{CF}^\infty}
\def\CFa {\operatorname{\widehat{CF}}}

\def\HFa {\operatorname{\widehat{HF}}}

\def\CFK {\operatorname{CFK}}

\def\CFKm {\operatorname{CFK}^-}
\def\HFKm {\operatorname{HFK}^-}
\def\CFKi {\operatorname{CFK}^\infty}

\DeclareMathOperator{\Spin}{Spin}
\DeclareMathOperator{\tb}{tb}
\DeclareMathOperator{\rot}{rot}
\def\spincs {\mathfrak{s}}
\def\X {\mathcal{X}}

\newcommand{\rotq}{rot_{\mathbb{Q}}}
\newcommand{\tbq}{tb_{\mathbb{Q}}}

\DeclareMathOperator{\PD}{PD}

\title{Surgeries on knots and tight contact structures}

\author{Zhenkun Li}
\address{Academy of Mathematics and Systems Science, Chinese Academy of Sciences, Beijing, China}
\email{zhenkun@amss.ac.cn}

\author{Shunyu Wan}
\address{Department of Mathematics, Georgia institute of technology, GA, USA}
\email{swan48@gatech.edu}

\author{Hugo Zhou}
\address{Department of Mathematics, University of Michigan, Ann Arbor, MI, USA}
\email{hugozhou@umich.edu}

\begin{document}
\maketitle
\begin{abstract}
For any knot $K$ in $S^3$ and any positive rational $r$, we show that smooth $(-r)$-surgery on $K$ always admits a tight contact structure. More specifically, the tightness is detected by the non-vanishing Heegaard Floer contact invariant.

  % We say a Legendrian knot $L$ in some contact $S^3$ is admissible if  $\tb(L)-\rot(L)=2g_3(L)-1$, and $\tb(L)\neq 1$. We show that if a admissible Legendrian $L$ has non-zero LOSS invariant then the contact invariant of Legendrian surgery on $L$ is non-zero; moreover, if two admissible Legendrian $L_1$ and $L_2$ have different LOSS invariants then the contact invariants of Legendrian surgery on $L_1$ and $L_2$ are different. 

%We also show that given any smooth knot $K$ in $S^3$ we can always find a family of admissible Legendrian representatives with non-vanishing LOSS invariants in some contact $S^3$. As applications, we show that any negative surgery on any knot in $S^3$ admits at least one tight contact structure with non-vanishing contact invariant. 

%for any positive integer $n$ if $K$ also satisfies $\tau(K)<g_3(K)$, then $-n$ surgery on $K$ admits a tight contact structure with non-vanishing contact invariant obtained by doing Legendrian surgery on a non-loose representative in some overtwisted $S^3$; moreover, if $K$ is non-fibered  $-n$ surgery on $K$ admits two tight contact structure with distinct non-vanishing contact invariants. This gives our first examples where Legendrian surgery on two different non-loose non-fibered knot being tight and different.
\end{abstract}
\section{Introduction}

The existence and nonexistence of tight contact structures on the $3$-manifold are interesting and important topics studied over the past thirty years. Etnyre-Honda \cite{EtnyreHondanonexistenceoftight} found the first example of a $3$-manifold that does not admit tight contact structure, and later Lisca-Stipsicz \cite{LiscaStipsicztightSeifert} extended their result and showed that a Seifert fiber space admits a tight contact structure if and only if it is not the smooth $(2n-1)$-surgery along the $T_{2,2n+1}$ torus knot for any positive integer $n$.

Surprisingly, since then no other example of a $3$-manifold without tight contact structure has been found. In particular, we do not know of a hyperbolic $3$-manifold that does not admit a tight contact structure, and it has become a central problem to determine whether such an example exists. 
%
%Moreover, people start to conjecture that hyperbolic $3$-manifold always admit a tight contact structure. 
%
%Toward the direction of this conjecture, 
Since the only examples of irreducible $3$-manifolds not admitting a tight contact structure are constructed by surgery on a knot in $S^3$, it is interesting to study if all such manifolds, except those mentioned above, admit a tight contact structure. Towards this goal, we are able to prove the following theorem.

\begin{theorem} \label{thm: negative surgery on knot is tight}
     For any knot $K\in S^3$ and $r\in \mathbb{Q}_{> 0}$, $S^3_{-r}(K)$ admits a tight contact structure with non-vanishing Heegaard Floer contact invariant.
\end{theorem}

Moreover, since $0$ surgery on a knot always admits a taut foliation \cite{GabaiFoliaationsIII}, hence a tight contact structure \cite{EliashbergThurstonConfoliations}, we also have the following corollary. 

\begin{corollary}
    For any knot $K\in S^3$ and $r\in \mathbb{Z}_{\geq 0}$, $S^3_{-r}(K)$ admits at least one tight contact structure.
\end{corollary}

\begin{remark}
    In particular, the above corollary says that every hyperbolic manifold that is a non-positive surgery along the knot in $S^3$ admits a tight contact structure.
\end{remark}

To better understand how the proof of Theorem \ref{thm: negative surgery on knot is tight} goes, let us first consider a related problem: given a Legendrian knot $L$, how does $L$ (or its complement) affect the contact structure obtained by the Legendrian surgery on $L$. More precisely, we can ask the following two questions: 

\begin{enumerate}
    \item [Q1.] If $L_1$ and $L_2$ are two non-Legendrian isotopic Legendrian knots in a contact $3$-manifold, when are the contact structures obtained by the Legendrian surgery on $L_1$ and $L_2$ different?
\item [Q2.] If $L$ is a non-loose (i.e., the complement is tight) Legendrian knot in some overtwisted contact $3$-manifold, when is the Legendrian surgery on $L$ tight?
\end{enumerate}

The question Q1 has been studied, for example, in \cite{BourgeoisEkholmEliashbergEffectofLegendriansurgery,EtnyreOnContactSurgery,CasalsEtnyreKegelSteintraces,WZnegativecontactsurgery}, while Q2 to the best of authors' knowledge has only been systematically studied in an upcoming paper \cite{EtnyreMinTosunVarvarezosPre}. In general, those are also very hard questions, and we are trying to answer the Heegaard Floer analog of those two questions when the underlying $3$-manifold is $S^3$; that is, how the Legendrian LOSS invariant \cite{LOSS} of $L$ affects the contact invariant \cite{OSHFandcontactstructures,HKMcontactclassinHF} of the Legendrian surgery on $L$. We first give the important definition. 

\begin{definition}
   An \emph{admissible} Legendrian knot $L$ is a Legendrian knot in some contact structure on $S^3$ such that $\tb(L)\neq 1$ and $\tb(L)-\rot(L)=2g(L)-1$, where $g(L)$ is the $3$-genus of the smooth knot type of $L$.
\end{definition}

Here are several remarks about this definition.

\begin{remark} 
\textbf{}
\begin{enumerate}
\item The condition on $\tb(L)\neq 1$ ensures that the Legendrian surgery on $L$ yields a rational homology sphere. 

\item If $L$ satisfies $\tb(L)-\rot(L)=2g-1$ then the LOSS invariant of $L$ lives in the top $g$ grading of $\HFKm(-S^3,K)$ \cite[Theorem 1.6]{OSct}. 
\item If $L$ is admissible then negative stabilization of $L$ is also admissible, except when $\tb(L)=2$, and if $\tb(L)=2$ then twice-negatively stabilized $L$ is also admissible.
\end{enumerate}
    
\end{remark}

For admissible Legendrian knots, we can answer $Q_1$ and $Q_2$ in the Heegaard Floer setting.

\begin{theorem} \label{thm: LOSS determines contact for admissible Legendrian}
 Let $L$ and $L'$ be two admissible Legendrian knots with the same knot type and same Thurston–Bennequin number.  
 \begin{enumerate}
     \item If $L$ and $L'$ have distinct LOSS invariants, then the Legendrian surgeries on $L$ and $L'$ have distinct contact invariants.
     \item If $L$ has non-vanishing LOSS invariant, then the Legendrian surgery on $L$ has non-vanishing contact invariant.
 \end{enumerate}
\end{theorem}

Note that Theorem \ref{thm: LOSS determines contact for admissible Legendrian} (1) recovers the results of \cite[Theorem 1.1]{WZnegativecontactsurgery}. Moreover, many known examples of Legendrian knots with same knot type, $\tb$, and $\rot$ but different LOSS invariants are admissible, for example see \cite[Theorem 5.7]{OSct} \cite[Theorem 5.3]{WanNaturalityofLOSSinvariant} \cite[Theorem 1.1]{FoldvariLegenriantwobridge}. Thus, Legendrian surgeries on any pair of those Legendrian knots give different contact manifolds with different contact invariants.

The second part of Theorem \ref{thm: LOSS determines contact for admissible Legendrian} is the key to proving Theorem \ref{thm: negative surgery on knot is tight}. However, to make it work, we need examples of admissible Legendrian with non-vanishing LOSS invariant. It turns out that every knot has admissible Legendrian representatives in some contact structure on $S^3$.

%The second part of the Theorem only interesting when $L$ is a admissible Legendrian knot with non-vanishing LOSS invariant in some overtwisted $S^3$, and the Theorems below show that there are actually many of such examples.

\begin{theorem}\label{thm: 2g-1 non-loose knot}
   For any knot $K\in S^3$ and $k\in \mathbb{Z}_{\geq 0}$, there exists some contact structure on $S^3$ such that $K$ has an admissible Legendrian representative $L^{-k}$ with $\tb(L^{-k})=-k$ and $\rot(L^{-k})=-k-2g+1$. Moreover, for each $k$, $L^{-k}$ has non-vanishing LOSS invariant in $\HFKm (-S^3,K,g(K))$ and $S_{-}(L^{-k})$; the negative stabilization of $L^{-k}$, is $L^{-k-1}$. In particular, when $\tau(K)<g(K)$, this contact structure is overtwisted and $L^{-k}$ is a  non-loose Legendrian knot.
\end{theorem} 

%The above theorem can be strengthened when the knot is non-fibered. 

%\begin{theorem}\label{thm: 2g-1 non-loose knot, 2}
 %   For any non-fibered knot $K\in S^3$ and $k\in \mathbb{Z}_{\geq 0}$, %there exists some contact structure on $S^3$ such that 
 %   $K$ has two families of admissible Legendrian representatives $L^{-k}$ and $S^{-k}$ (which may or may not in the same contact $S^3$) satisfying the properties in Theorem \ref{thm: 2g-1 non-loose knot} with distinct LOSS invariant.
%\end{theorem}
    
Now by combining Theorem \ref{thm: 2g-1 non-loose knot} with Theorem \ref{thm: LOSS determines contact for admissible Legendrian} (2), we can perform Legendrian surgeries on those families of admissible Legendrian knots, and we immediately obtain the following corollary, which implies the integer version of Theorem \ref{thm: negative surgery on knot is tight}.

\begin{corollary}
    
\label{thm: Legendrian surgery with non-vanishing contact invariant}

For any knot $K\in S^3$ and $n\in \mathbb{Z}_{> 0}$, $S^3_{-n}(K)$ admits at least one tight contact structure with non-vanishing contact invariant. %If $K$ is also non-fibered, then $S^3_{-n}(K)$ admits at least two tight contact structures with distinct non-vanishing contact invariants. 
Moreover, if $\tau(K)<g(K)$ those contact structures are coming from Legendrian surgeries on non-loose Legendrian representatives of $K$.
\end{corollary}

Using the Ding-Geiges-Stipsicz \cite{DGS} algorithm for negative rational contact surgeries and the fact that a Legendrian surgery preserves the non-vanishing of contact invariant (or see \cite[Section 5.3]{WZnegativecontactsurgery}), we can easily upgrade the above theorem to any non-negative rational number, and hence obtain Theorem \ref{thm: negative surgery on knot is tight}. Thus, the main goal of the rest of the paper is to prove Theorem \ref{thm: LOSS determines contact for admissible Legendrian} and Theorem \ref{thm: 2g-1 non-loose knot}.

Before we move on, some more concrete examples we can consider are the Legendrian non-loose negative torus knots which are classified in \cite{EtnyreMinMukherjeeNonloosetorus}. By performing Legendrian surgeries on certain non-loose representatives, we also obtain the following corollary. (Compare the result in \cite[Theorem 1.12]{min2024contactinvariantsborderedfloer})   %Combine the result that $0-$surgery on any knot in $S^3$ admits tight contact structure by Gabai\cite{GabaiFoliaationsIII} and Eliashberg-Thurston \cite{EliashbergThurstonConfoliations} we can conclude the following.  

%Once we have proved Theorem \ref{thm: 2g-1 non-loose knot}, another key ingredient we need to proof Theorem \ref{thm: Legendrian surgery with non-vanishing contact invariant} is the injectivity theorem (Theorem \ref{thm: leq -2g-1 surgery}) for the map that sends the LOSS invariant of the dual knot to contact invariant in $S^3_n(K)$. 

 %Note that the non-simple Legendrian examples found in \cite{OSct}, \cite{FoldvariLegenriantwobridge},and \cite{WanNaturalityofLOSSinvariant} all satisfy the condition in the above theorem, so negative contact surgery on those non-simple representatives are almost always distinct. Also notice that the above theorem does not require the Legendrian knots are in the standard tight $S^3$. Thus, in the same vibe as Theorem \ref{thm: r Legendrian surgery}, doing Legendrian surgery on certain non-loose negative torus knot (classified in \cite{EtnyreMinMukherjeeNonloosetorus}) we obtain the following corollary. 

\begin{corollary}
    If $K$ is a negative torus knot, then for any integer $n$, $S^3_n(K)$ admits a tight contact structure.
\end{corollary}
\begin{remark}
    This result follows from \cite{LiscaStipsicztightSeifert}, but we indicate how our arguments work for surgery on non-loose knots.
\end{remark}

\begin{proof}
    Let $T$ be the transverse representative of $K$ that corresponds to the binding of open book supported by $K$. In \cite{EtnyreMinMukherjeeNonloosetorus}, it has been shown that for all $i\in \mathbb{Z}$, there exists $L^i$ such that $\tb(L^i)=i$, $\rot(L^i)=i-2g(K)+1$, negative stabilization of $L^i$ is $L^{i-1}$, the transverse push-off of $L_i$ is $T$, and $L^i$ has non-vanishing LOSS invariant\cite{VelaVicktransverseinvariantofopenbook}. Thus, Theorem \ref{thm: LOSS determines contact for admissible Legendrian} implies that the Legendrian surgery on $L^i$ is tight except possibly when $i=1$ ($n=0$). However, again the $n=0$ case follows from the fact that $S^3_0(K)$ always admits a taut foliation, hence a tight contact structure\cite{GabaiFoliaationsIII} \cite{EliashbergThurstonConfoliations} . 
    \end{proof}
 
\subsection*{Organization} In Section \ref{sec:hfpreli} we review the preliminaries for the Heegaard Floer homology. In Section \ref{sec:contactpreli} we discuss the LOSS invariant of rationally null-homologous Legendrian knot, and the naturality result of both LOSS and contact invariant. In Section \ref{sec: Negative contact surgeries on thin knots} we prove Proposition \ref{Prop: inject at top grading}, an important technical result in Heegaard Floer homology. In Section \ref{sec: proof of LOSS determines contact for admissible Legendrian} we prove Theorem \ref{thm: LOSS determines contact for admissible Legendrian} using  Proposition \ref{Prop: inject at top grading}. In Section \ref{sec: construction of admissible Legendrian} we review  some background on contact suture manifold and then give the proof of Theorem \ref{thm: 2g-1 non-loose knot}. Lastly, in Section \ref{sec: proof of negative surgery on knot is tight} we prove  Theorem \ref{thm: negative surgery on knot is tight}.
\subsection*{Acknowledgement} The authors would like to thank John Etnyre and Jen Hom for helpful discussions. HZ is supported by an AMS-Simons travel grant.

\section{Heegaard Floer Homology Preliminaries} \label{sec:hfpreli}
We briefly review Heegaard Floer homology in this section. The goal is to introduce the dual knot surgery formula in Section 
\ref{subsec: dualknot}, and to define the contact and the Loss invariant in Section \ref{sec:contactpreli}.
\subsection{Heegaard Floer invariants for three-manifolds and knots}
Given a closed oriented $3$-manifold $Y$ with a fixed basepoint $w$, described by a  tuple $(\Sigma, \alpha,\beta, w)$, where $\Sigma \subset Y$ is a smoothly embedded genus-$g$ surface, each of $\alpha$ and $\beta$ is a collection of $g$ pair-wise disjoint simple closed curves on $\Sigma$,  Ozsv\'{a}th-Szab\'{o} \cite{OSht} defined $\CF^\circ (Y)$ with four different flavors $\circ = \wedge,+,-$ and $\infty$, called the \emph{Heegaard Floer chain complex}. The generators for $\CF^\circ (Y)$ are the intersections of two Lagrangian tori constructed from $\alpha$ and $\beta$ curves in $\operatorname{Sym}^g(\Sigma)$ and the differentials are given by counting holomorphic disks. The invariant  $\CFa(Y)$ is a chain complex over $\F,$ where $\F = \Z/2\Z$ is the field with two elements; the invariants $\CFm(Y)$ and $\CFp(Y)$ are chain complexes over the ring $\F[U]$ and $\CFi(Y)$ is a chain complex over the ring $\F[U,U^{-1}];$ each complex admits a Maslov grading.    By taking the homology of $\CF^\circ (Y)$, we obtain the module $\HF^\circ(Y)$, called the \emph{Heegaard Floer homology of $Y$}.  

Given a knot $K\subset Y$,  the Heegaard Floer complexes admit a refinement to a knot invariant \cite{OSknot,rasmussen_thesis}.   Each pair $(Y,K) $ is described by $(\Sigma, \alpha,\beta, w,z)$ where $z$ is an extra basepoint, which allows
us to  impose an $(i,j)$ double-filtration over the original chain complex $\CF^\infty (Y)$. The resulting chain complex over $\F[U,U^{-1}]$ is graded and doubly filtered, denoted by $\CFKi(Y,K)$. It is  called the \emph{full knot Floer chain complex}, since other versions of the knot invariants can be obtained as a sub- or quotient complex of it. 

The complex  $\CFKi(Y,K)$ consists of formal elements $x=[\x,i,j],$ where $\x$ is an intersection point of the Lagrangian tori, where
$i$ is the \emph{algebraic filtration level}, $j$ is the \emph{Alexander filtration level}, and $U$ acts by $Ux = [\x,i-1,j-1]$. The \emph{Alexander grading} of $x=[\x,i,j]$ is defined to be $j-i$. The differential in $\CFKi(Y,K)$ counts the number of times a disk intersects the basepoints: $i$-filtration is decreased by $1$ each time the disk crosses $w$, and $j$-filtration is decreased by $1$ each time the disk crosses $z$. 

The group of $\Spin^c$ structures of a closed $3$-manifold $Y$ admits a non-canonical isomorphism  $\Spin^c(Y) \cong H_1(Y;Z)$. The Heegaard Floer chain complexes and knot Floer chain complexes split over the $\Spin^c$ structures. Namely, 
\[
\CF^\circ(Y) = \bigoplus_{\spincs \in \Spin^c(Y)} \CF^\circ(Y,\spincs)  \hspace{3em}  \CFK^\circ(Y,K) = \bigoplus_{\spincs \in \Spin^c(Y)} \CFK^\circ(Y,K,\spincs).
\]
Let $C= \CFKi(Y,K).$ We denote a sub/quotient complex of $C$  by $C\{S\}:= \{[\x,i,j]\in C \mid  (i,j)\in S\}$ for some $S$. 
There are relations between different complexes.  If we do not allow disks to cross $w$ basepoint, and also forget the filtration coming from the $z$ basepoint, we obtain
\[\CFa(Y,K) =  C\{i=0\}.\] 
Define 
\[
  \CFKm_s(Y) = C\{i\leq 0, j=s\} \text{and} \quad \CFKm(Y,K) =\bigoplus_{s\in \Z}  \CFKm_s(Y,K)
\]
where elements inherit their Alexander grading from $\CFKi(Y,K)$. Define
\[
  \HFKm_s(Y,K) =H_*(\CFKm_s(Y,K)) \quad \text{and} \quad \HFKm(Y,K) =\bigoplus_{s\in \Z}  \HFKm_s(Y,K).
\]
 Acting by $U$ sends $\CFKm_s(Y,K)$ to  $\CFKm_{s-1}(Y,K)$ which induces a map on $\HFKm (Y,K)$. Setting $U=1$ in $\CFKm (Y,K)$, we obtain $C\{j=0\}$. 
 
In \cite{LOSS}, Lisca-Ozsv\'ath-Stipsicz-Szab\'o considered a map

\begin{equation} \label{equ: g map}
    g:\CFKm(Y,K) \rightarrow \CFa(Y)
\end{equation}
by setting $U=1$, as above,  then swapping the role of $z$ and $w$, and thereby identifying $C\{j=0\}$ with $C\{i=0\}=\CFa(Y)$.
% We can view this map in the following way: If $\CFa(Y)$  is obtained from the data $(\Sigma, \alpha, \beta, w)$   and the extra basepoint $z$ provides the knot filtration in $\CFKi(Y,K)$, map $g$ can be thought as  first interchanging the role (not the location) of the  basepoints  then discarding  $z$; alternatively we can view $g=p \circ g_0\circ i,$ where $i: \CFKm(Y,K)\hookrightarrow \CFKi(Y,K)$ is the inclusion, $g_0: \CFKi(Y,K) \rightarrow \CFKi(Y,K)$ is the map induced by exchanging the role of $w$ and $z$ and $p:\CFKi(Y,K) \rightarrow \CFa(Y)$  is the natural projection. Both the maps  $i$ and $p$ are standard, so we would like to further analyze the map $g_0$. On the algebraic level, 
% under the map $g_0$ there is a natural identification of $\CFKi$ to itself, where  all the generators of $\CFKi$  remain unchanged, all the $z$-differentials become $w$-differentials and vice versa. Moreover, $g_0$ changes 
This map changes the $\Spin^c$ structure in the following way. 
% When $K$ is an oriented rationally null-homologous knot in a rational homology sphere,
If $x\in \CFKi(Y,K,\underline\spincs)$ where $\underline\spincs \in \underline\Spin^c(Y,K)$ is a relative $\Spin^c$ structure, by swapping $z$ and $w$, the relative $\Spin^c$  structure of $x$ changes from  $\underline\spincs$ to $J(\underline\spincs)-\PD(\mu)$ where $J$ denotes the $\Spin^c$ conjugation and  $\mu$ is the right-handed meridian of $K$. (See the proof of \cite[Proposition 8.2]{OSHololinkinvariants} and \cite[Equation 2.9]{HeddenLevine}. We are following the convention from \cite{HeddenLevine}.) In summary, $g$ splits over the  $\Spin^c$  structure as follows \begin{equation}  \label{eq: cfkminus_to_cfhat} g:\CFKm(Y,K,G_{Y,K}(\underline\spincs)) \rightarrow \CFa(Y, G_{Y,K}(J(\underline\spincs)-\PD(\mu)))
\end{equation}
where $G_{Y,K}: \underline\Spin^c(Y,K) \rightarrow \Spin^c(Y)$ is the map defined by  Ozsv\'ath-Szab\'o in \cite[Section 2.2]{rational}. 
% Note that $G_{Y,K}$ depends on the orientation of $K.$

When $K$ is null-homologous, $\underline\Spin^c(Y,K) \cong \Spin^c(Y) \oplus \Z$ where $\Z$ is generated by $\PD(\mu)$ and $G_{Y,K}$ is a projection on to the first summand. Therefore if $ G_{Y,K}(\underline\spincs) = \spincs  $, we simply have $G_{Y,K}(J(\underline\spincs)-\PD(\mu)) = J(\spincs)$. For integer homology spheres, e.g. $S^3$, since there is a unique spin$^c$ structure, $\spincs$ and $J(\spincs)$ are the same.

\subsection{Dual knot surgery formula} \label{subsec: dualknot}
In \cite{HeddenLevine}, Hedden and Levine defined a mapping cone formula that for $n \neq 0$ computes  $\CFKi(S_n^3(K),\mu)$, where $\mu$ is the meridian (or dual knot) of $K$ in the surgery. 
Define $\X^{\infty}_K(C)$ to be the mapping cone of
 \begin{align} \label{eq:x_infinity2}
      \bigoplus^{g}_{s=-g+1}A_s(C) \xrightarrow{v^{\infty}_s+h^{\infty}_s} \bigoplus^{g}_{s=-g+n+1}B_s(C).
\end{align}
where each $A_s$ and $B_s$ is a copy of $\CFKi(S^3,K).$
Define the double filtration $(\II, \JJ)$ and the Maslov grading over $\X^{\infty}_K(C)$ as follows.
\begin{align}
\intertext{For $[\x,i,j] \in A_{s}(C)$,}
\label{eq: filtration_s3_1}
 \II([\x,i,j]) &= \max\{i,j-s\} \\
 \label{eq: filtration_s3_2}
 \JJ([\x,i,j]) &= \max\{i-1,j-s\} + \frac{2s+n-1}{2n} \\
\label{eq: grt-def-A} \gr([\x,i,j]) &= \widetilde{\gr}([\x,i,j]) + \frac{(2s-n)^2}{4n} + \frac{2-3\operatorname{sign}(n)}{4}  
\intertext{and for $[\x,i,j] \in B_{s}(C)$,}
 \label{eq: filtration_s3_3}
 \II([\x,i,j]) &= i \\
  \label{eq: filtration_s3_4}
 \JJ([\x,i,j]) &= i-1 + \frac{2s+n-1}{2n}
  \\ \label{eq: grt-def-B} \gr([\x,i,j]) &= \widetilde{\gr}([\x,i,j]) + \frac{(2s-n)^2}{4n} + \frac{-2-3\operatorname{sign}(n)}{4}
\end{align}
where $ \widetilde{\gr}$ indicates the Maslov grading  in the original complex.
\begin{theorem}[Theorem 1.1 in \cite{HeddenLevine}]
    The complex $\CFKi(S^3_{n}(K),\mu)$ is filtered chain homotopy equivalent to $\X_K^{\infty}(\CFKi(S^3,K))$.
\end{theorem}
Collapsing the $\JJ$-filtration in $\X_K^{\infty}(C)$ recovers the  complex $\X^{\infty}(C)$ in the original construction \cite{OSIntegersurgery} by Ozsv\'{a}th  and Szab\'{o}, which is homotopy equivalent to $\CFi(S^3_n(K))$. 

 Consider the cobordism $W'_n: S^3_n(K) \to S^3$  obtained by turning around the two-handle cobordism $-W_n: -S^3 \to -S^3_n(K)$. Denote by $\Tilde{F}$  a capped Seifert surface of $K$ in $W'_n.$ 
By \cite[Section 2.4]{OSIntegersurgery},  
given $\spincs \in\Spin^c(S^3_n(K)) $, if we let $\underline\spincs$ denote the extension of $\spincs$ to $\Spin^c(W'_n(K))$, then 
$\Spin^c(S^3_n(K))  $ is identified with $\Z/n\Z$ through the map $\spincs \mapsto i \in \Z/n\Z$, where
\[\langle c_1(\underline\spincs),[\Tilde{F}] \rangle  - n \equiv 2i  \quad (\operatorname{mod} 2n ). \]

By \cite[Theorem 4.1]{OSIntegersurgery}, under the above identification, each $A_s, B_s$ in $\X_K^{\infty}(C)$ is in spin$^c$ structure $s \in \Z/n\Z \cong \Spin^c(S^3_n(K))  $. In other words,  for $s \in \Z/n\Z$, the mapping cone of
\[
\bigoplus_{-g+1 \leq s+n\ell \leq g}  A_{s+n\ell}(K) \xrightarrow{v^{\infty}_s+h^{\infty}_s} \bigoplus_{-g+n+1 \leq s+n\ell \leq g}  B_{s+n\ell}(K)
\]
is homotopy equivalent to the summand $\CFKi(S^3_n(K),\mu,s).$ Moreover, the spin$^c$ conjugation $J$ sends $s$ to $-s.$
Therefore for the manifold-knot pair $(S^3_n(K),\mu),$ the map \eqref{eq: cfkminus_to_cfhat} can be written as
  \begin{equation} 
g:\CFKm(S^3_n(K),\mu,s) \rightarrow \CFa(S^3_n(K),1-s)
    \end{equation}
for $s \in \Z/n\Z \cong \Spin^c(S^3_n(K)),$ which  induces a map on homology 
  \begin{equation} \label{eq: g to 1-g map}
f:\HFKm(S^3_n(K),\mu,s) \rightarrow \HFa(S^3_n(K),1-s).
    \end{equation}
 
%\begin{remark} \label{MCG action on LOSS}
% The LOSS invariants are actually only well defined up to sign, and up to the action of the mapping class group  on $(Y,L)$ \cite{OSct} (that is, the group of isotopy classes of diffeomorphisms of $Y$ fixing $L$). We denote by $[\mathfrak{L}]\in  \HFKm (-Y,L)/\pm MCG(Y,L)$  the image of $\mathfrak{L}$ when we quotient out these actions. Moreover by the result in \cite{JuhaszThurstonZemkeNaturalityaandMCGinHF} if two Legendrian knot $L_1$ and $L_2$ are both smoothly isotopic to $K$, and $[\mathfrak{L_1}] \neq [\mathfrak{L_2}]$, we could actually say that they represent different orbits in $\HFKm (-Y,K)$ when we do smooth calculation. 
%\end{remark}

\section{LOSS invariant for rationally null-homologous  Legendrian}\label{sec:contactpreli}
In this section we review the definitions and properties regarding rationally null-homologous Legendrian LOSS invariant that we will use later.

\subsection{LOSS invariant for rationally null-homologous Legendrian knots}

Let $Y$ be a  rational homology sphere, and $K$ be an oriented rationally null-homologous knot with order $p$, which means that $[K]$ in $H_1(Y,\mathbb{Z})$ has order $p$.

Initially,  Lisca-Ozsv\'ath-Stipsicz-Szab\'o defined the ``LOSS invariant'' $\mathfrak{L}(L) \in \HFKm(-Y,L)$ for null-homologous Legendrian knots \cite{LOSS}. 
In \cite{BakerEtnyreRationallinking} Baker and Etnyre  defined the rational Thurston-Bennequin number and the rational rotation number for rationally null-homologous Legendrian knots. Using the fact that we may realize any Legendrian knots on the page of the open book compatible with the contact structure \cite{EtnyreLecturesonopenbook},
% However, after Ozsv\'ath-Szab\'o defined the knot Floer for rationally null-homologous knot \cite{rational} (see also \cite{HeddenLevine}) together with the fact that we may realize any Legendrian on the page of the open book compatible with the contact structure \cite{EtnyreLecturesonopenbook},
 we observe that the definition of the LOSS invariant extends naturally to the rationally null-homologous Legendrian knots. Moreover, all important properties of the null-homologous case are still preserved for the rationally null-homologous case. However, there is one important difference of the LOSS invariant between the null-homologous Legendrian knots and the rationally null-homologous Legendrian knots, that is the $\Spin^c$ structure the LOSS invariant lives in. 

In \cite{LOSS}, Lisca-Ozsv\'ath-Stipsicz-Szab\'o introduced an important map that relates the LOSS invariant to the contact invariant \[g:\CFKm(Y,K) \rightarrow \CFa(Y),\]
and note that this is exactly the same $g$ map (\ref{equ: g map}) we discussed in the previous section.
Using exactly the same argument as in \cite{LOSS},  the following lemma  relates the LOSS invariant of rationally null-homologous Legendrian knots  to the contact invariant.

\begin{lemma}[\cite{LOSS}] \label{lem: f maps rational LOSS to Contact invariant}
    Let $L$ be an oriented rationally null-homologous Legendrian knot in a contact 3-manifold $(Y,\xi)$. Then there is a map \begin{equation} \label{LOSS to contact invariant}
        f: \HFKm(-Y,L,G_{Y,K}(\underline\spincs)) \rightarrow \HFa(-Y,G_{Y,K}(J(\underline\spincs)-\PD(\mu)))
    \end{equation} for each $\underline\spincs \in \underline\Spin^c(Y,K)$,  which has the property that $$f(\mathfrak{L}(L))=c(\xi),$$
where $c(\xi)\in \HFa(-Y)$ is the Heegaard Floer contact invariant defined by Ozsv\'{a}th-Szab\'{o} \cite{OSHFandcontactstructures} and Honda-Kazez-Mati\'c \cite{HKMcontactclassinHF}.
\end{lemma}

% There are some more properties of the LOSS invariant of rationally null-homologous knot that we are going to use later. We start with the setting that
Next, suppose that $L$ is an order-$p$ rationally null-homologous Legendrian in a rational homology sphere $Y$, and denote $\tbq(L)$ and $\rotq(L)$ to be the rational Thurston-Bennequin number and rational rotation number of $L$ respectively. 

\begin{theorem}\cite[Theorem 1.6]{OSct} \label{thm: A grading of LOSS}
$$2A(\mathfrak{L}(L))=\tbq(L)-\rotq(L)+1,$$
where $A(\mathfrak{L}(L))$ is the Alexander grading of $(\mathfrak{L}(L))$. Note that $A(\mathfrak{L}(L))\in \frac{1}{2p}\mathbb{Z}$.    
\end{theorem}

The proof of the above theorem follows directly from the definition of the Alexander grading of rationally null-homologous knot and the proof of \cite[Theorem 4.1]{OSct}. All the related lemmas can be rewrote in the rationally null-homologous case (the proofs are straightforward analogs to the null-homologous case). 

Since the negative stabilization for any Legendrian knot appears in the same way on the open book, we still have the following stabilization theorem for the LOSS invariant.

\begin{theorem}[\cite{LOSS}] \label{LOSS invariant under negative stabilization}
Suppose that $L$ is an oriented rationally null-homologous Legendrian knot and denote the negative  stabilizations of $L$ by $L^-$.  Then,
$\mathfrak{L}(L^-) = \mathfrak{L}(L)$.
\end{theorem}

\subsection{Naturality  under contact surgery for rationally null-homologous Legendrian knots.} \label{subsec: Naturality}

We first discuss the naturality results for the contact invariant. 
\begin{theorem}[{\cite[Theorem 6.3]{WanNaturalityofLOSSinvariant}}] \label{thm: same Spin^c for Legendrian with classical invariants}
    Let $L$ be an oriented rationally null-homologous Legendrian knot in a contact rational homology sphere $(Y,\xi)$ with non-vanishing contact invariant $c(\xi)$. Let $0<n \in \mathbb{Z}$ be the contact surgery coefficient, $Y'$ be the manifold after contact $(+n)$-surgery on $L$, and let $W:Y \rightarrow Y'$ be the corresponding  surgery cobordism, and consider $\xi_{n}^-(L)$ on $Y'$. There exists a $\Spin^c$ structure $\mathfrak{s}$ on $W$ such that the homomorphism 
        
        $$F_{-W,\mathfrak{s}}: \HFa(-Y) \rightarrow \HFa(-Y')$$ 
        induced by $W$ with its orientation reversed satisfies 
        $$F_{-W,\mathfrak{s}}(c(\xi))= c(\xi_n^-(L)).$$
Moreover, if $Y'$ is also a rational homology sphere. Then the $\Spin^c $ structure $\mathfrak{s}$ has the property that $$ \langle c_1(\mathfrak{s}),[\Tilde{F}] \rangle = 
      y(rot_\mathbb{Q}(L)+n-1) $$
 where $y$ is the order of $[L]$, $F$ is a rational Seifert surface for $L$ and $\Tilde{F}$ is the ``capped off" surface of $F$.
\end{theorem}

We also have the parallel naturality theorem for the LOSS invariant, which will also be used later.

\begin{theorem}[{\cite[Theorem 1.1]{OzsvathStipsiczContactsurgeryandtransverseinvariant}}]{\label{Naturality of LOSS}}
Let $L,S \in (Y, \xi)$ be two disjoint oriented Legendrian knots in the contact 3-manifold $(Y, \xi)$ with $L$ rationally null-homologous. Let $(Y',\xi_1)$ denote the 3-manifold with the associated contact structure obtained by performing contact $(+1)$-surgery along $S$, and let $L'$ denote the oriented Legendrian knot which is the image of $L$ in $(Y',\xi_1)$. Moreover, suppose that $L'$ is rationally null-homologous in $Y'$. Let $W$ be the 2-handle cobordism from $Y$ to $Y'$ induced by the surgery, and let
\begin{equation*}
    F_{-W,\mathfrak{s}}: \HFKm (-Y,L) \rightarrow \HFKm(-Y',L')
\end{equation*}
be the homomorphism in knot Floer homology induced by $-W$, the cobordism with reversed orientation, for  $\mathfrak{s}$ a $\Spin^c$ structure on $-W$. Then

\begin{enumerate}
    \item \label{it: naturality1}  there is a unique choice of \s \ for which 
\begin{equation*}
    F_{-W,\mathfrak{s}}(\mathfrak{L}(L))=\mathfrak{L}(L')
\end{equation*}
holds, and for any other $\Spin^c$ structure \s \  the map $F_{-W,\mathfrak{s}}$ is trivial on $\mathfrak{L}(L)$. 

\item\label{it: naturality2}  \cite[Proposition 1.4]{WanNaturalityofLOSSinvariant} If $S$ is rationally null-homologous with order $y$ and both $Y$ and $Y'$ are rational homology spheres, then $\mathfrak{s}$ has the property that $$ \langle c_1(\mathfrak{s}),[\Tilde{Z}] \rangle =y\rot(S)$$
where $Z$ is a Seifert surface for $S$ and $\Tilde{Z}$ is the result of capping off $Z$ with the core of the handle in $W$.
\end{enumerate}

\begin{remark}
   Again, in {\cite[Theorem 1.1]{OzsvathStipsiczContactsurgeryandtransverseinvariant}} $L$ was assumed to be null-homologous but the proof is exactly the same for rationally null-homologous.  
\end{remark}

\end{theorem}

As an immediate corollary of the above theorem, we have the following. 

\begin{corollary}\label{cor: Legen surgery on distinct loss give distinct loss}
    Let $L_1$ and $L_2$ be two rationally null-homologous Legendrian knots with the same knot type $K$, $\tb$, and $\rot$ in some contact rational homology sphere $(Y,\xi)$. Let $(Y_i,\xi_i)$ be the contact $3$-manifold obtained by performing the Legendrian surgery on $L_i$ respectively for $i=1,2$, and let 
$L_i'$ be the Legendrian push-off of $L$ in $(Y_i,\xi_i)$. Moreover if $Y_i$ is also a rational homology sphere, then $\mathfrak{L}(L_1)\neq \mathfrak{L}(L_2)$ implies $\mathfrak{L}(L_1')\neq \mathfrak{L}(L_2')$.
\end{corollary} 

\begin{proof}
    Follow directly from the above theorem by viewing a Legendrian surgery from $(Y,\xi)$ to $(Y_i,\xi_i)$ as contact $(+1)$-surgery from $(Y_i,\xi_i)$ to  $(Y,\xi)$. 
\end{proof}

\section{Dual knot and negative contact surgeries } \label{sec: Negative contact surgeries on thin knots}
In this section, we prove an important technical result using the dual knot surgery formula, then apply it to the contact setting.
\subsection{Injectivity and the dual knot complex.} 
 We will show that for positive integer surgeries,   the map $f$ in \eqref{eq: g to 1-g map} 
 is always injective at the top Alexander grading of $\HFKm(S^3_n(K),\mu, g)$ where $g=g(K)$.  
\begin{definition}
    % Given a chain complex  $C\simeq\CFKi(S^3,K)$, we can further fix a basis $B$, which generates $C$ over $\F[U,U^{-1}]$. If $B\subset C\{i=0\}$, we call it a basis for $\HFa(S^3).$ 
    We say that a knot Floer chain complex $C\simeq\CFKi(S^3,K)$ is reduced if every differential strictly decreases either the $i$- or $j$- filtration level.  There is a standard procedure for obtaining a reduced complex, see, for example, \cite[Proposition 11.57]{LOTBordered}.  Given a generator $w_1\in C$ such that $\partial w_1 $ is in the same filtration as $w_1$, quotienting out the subcomplex spanned over $\F[U,U^{-1}]$ by $\{w_1,\partial w_1\}$ induces a homotopy equivalence of $C$. One can find a set of generators $\{w_k\}_{1\leq k \leq m}\subset C$ (e.g. by repeating the previous procedure) such that $C$ quotioning $\{w_k,\partial w_k\}_{1\leq k \leq m}$ is reduced. Alternatively, this can be viewed as a change of basis followed by a cancellation of pairs of generators.   
    % We say a basis $B$ is vertically simplified if $B=\{ x_i, y_i \}_{1\leq i \leq n} \cup \{z\}$ such that the only non-trivial vertical differentials in $B$ are $\partial^{\operatorname{vert}}x_i = y_i$ for $1\leq i\leq n$. We call $z$ the distinguished generator. 
\end{definition}
We first  determine the top Alexander grading of $\HFKm(S^3_n(K),\mu, g)$. This will also help us locate the LOSS invariants later on.
\begin{lemma} \label{lem: top A grading}
    For any positive integer $n$, the top Alexander grading of $\HFKm(S^3_n(K),\mu, g)$ is $\frac{2g+n-1}{2n}$.
\end{lemma}
\begin{proof}
Suppose $\CFKi(S^3,K)$ is reduced.
By the discussion at the end of Section \ref{subsec: dualknot}, we have  
\begin{equation}
\label{eq:spincgmappingcone}
\CFKi(S^3_n(K),\mu, g) \simeq
\bigoplus_{\ell=0}^{\floor{\frac{2g-1}{n}}}  A_{g-n\ell}(K) \xrightarrow{v^{\infty}_s+h^{\infty}_s} \bigoplus_{\ell=0}^{\floor{\frac{2g-1}{n}}-1}  B_{g-n\ell}(K). 
\end{equation}
Endow each $A_s(K)$ with the $(i,j)$ filtration coming from $\CFKi(S^3,K)$.
    By \eqref{eq: filtration_s3_1}, the top Alexander grading of $\CFKi(S^3_n(K),\mu, g)$ is located in $\{i=0,j=g\} \subset A_s(K).$ 
    There is at least one element living in that filtration by \cite{genus}; call it $y.$ We have that $(\II,\JJ)(y)=(0,\frac{2g+n-1}{2n})$, and the mapping cone differential of $y$ consists of two parts: the internal differential of $A_s(K)$ and $v^\infty_g(y)$, both strictly lower the $\JJ$ filtration. The result follows.
\end{proof}

%     The symmetry of the knot Floer complex \cite[Proposition 8.2]{OSHololinkinvariants} induces a symmetry on the dual knot mapping cone (See for example \cite[Proposition 3.2]{HSZ_PL}.) In particular, the map $f$ in Lemma \ref{lem: f maps rational LOSS to Contact invariant} specializes to 
%     \begin{equation} \label{eq: g to 1-g map}
% f:\HFKm(S^3_n(K),\mu,g) \rightarrow \HFa(S^3_n(K),1-g).
%     \end{equation}
% It is easy to see under the identification of the $Spin^c$ in $S^3_n(K)$ structure and $\mathbb{Z_n}$ that if $\spincs=g(K)$, then $J(\spincs)-PD(\lambda)=1-g(K)$, where here we denote $\lambda$ to be the meridian of the dual knot $\mu$ and the reason why it is $1-g(K)$ rather than $-1-g(K)$ is that the orientation convention used by Hedden-Levine (which we are using) is opposite to Ozsv\'{a}th  and Szab\'{o}. The next goal is to show the injectivity of map $G$ at the top grading.   
We are ready to prove the following result:
\begin{proposition} \label{Prop: inject at top grading}
   For any positive integer $n$, the map \[f:\HFKm(S^3_n(K),\mu,g) \rightarrow \HFa(S^3_n(K),1-g)\] is injective at the top Alexander grading. 
\end{proposition}
\begin{proof}
We first consider the case when $n>2g$ and  $\CFKi(S^3_n(K),\mu,g) \cong A_g(K)$. Let $C= A_g(K)$,  which we can assume to be reduced. Endow $C$ with the $(i,j)$ filtration coming from $\CFKi(S^3,K).$
 We can also view $C$ as a complex $C^\mu$ with the filtration given by \eqref{eq: filtration_s3_1} and \eqref{eq: filtration_s3_2}.   A filtered change of basis of $C$ is necessarily a filtered change of basis of $C^\mu$, but not vice versa.  Fix a subcomplex $W \subset C^\mu$ spanned by some $\{w_k,\partial w_k\}_{1\leq k \leq m}$ where each pair $w_k$ and $\partial w_k$ have the same $(\II,\JJ)$ filtration, such that $C^\mu/W$ is reduced. Suppose, for contradiction, that the map \[f:\HFKm(S^3_n(K),\mu,g) \rightarrow \HFa(S^3_n(K),1-g)\] is not injective at the top grading. 
 Note that the associated graded complex with the top grading $\{\JJ = \frac{2g+n-1}{2n}\}  = \operatorname{max}\{i-1,j-g\}$ intersects $\{\II=\ell\}$ only when $\ell =0,1$.
 It follows that there are $\overline{x},\overline{y} \in C^\mu/W$ with $(\II,\JJ)(\overline{y}) = (0,\frac{2g+n-1}{2n})$ and $(\II,\JJ)(\overline{x}) = (1,\frac{2g+n-1}{2n})$ such that $\overline{y}$ has nonzero coefficient in $ \partial \overline{x}$, or equivalently, as elements in $C$,  
 $ \overline{y} +w'$ has  nonzero coefficient in $\partial (\overline{x} + w) $   for some $w,w' \in W$. 
 By  \eqref{eq: filtration_s3_1} and \eqref{eq: filtration_s3_2}, $\overline{y}$ must have $(i,j)$ filtration $(0,g)$. Since $\overline{y}$ has nonzero coefficient in $\partial (\overline{x} + w)$, and $(\II,\JJ)(\overline{x}) = (1,\frac{2g+n-1}{2n})$,  there must be a nontrivial term in $\overline{x} + w$ with $(i,j)$ filtration $(1,g)$. Since $(0,g)$ (resp.~$(1,g)$) is the highest $(i,j)$ filtration in $\{\II=0\}$ (resp.~$\{\JJ = \frac{2g+n-1}{2n}\}$), if we let $y = \overline{y} +w'$ (resp.~$x = \overline{x} +w$), it is a filtered element in $C$ with $(i,j)$ filtration $(0,g)$ (resp.~$(1,g)$). Next we perform a series of filtered changes of basis on $C$ as follows. 

 First, notice that any element in $\{i=1,j\leq g\}$ has the same $(\II,\JJ)$ filtration as $x$. Since $x$ is nontrivial in $C^\mu/W$, we can find some $w''\in \{i=1,j\leq g\}$ (potentially $w''=0$) such that no elements in $\{i=1,j\leq g\}$ have nontrivial coefficient in $\partial (x+w'')$. We perform a filtered change of basis $x \mapsto x+w''$; by abuse of notation we denote the resulting generator also by $x$. Now any nontrivial term in $\partial x$ strictly decreases the $i$ filtration. By a filtered change of basis, we can add all other terms to $y$. Denote the resulting generator by $y$. We now have $\partial x = y.$ It follows that $\partial y = \partial^2x=0$

 Next, suppose that $x$ has nonzero coefficient in $\partial z$ for some $z\in C$. By $\partial^2 z=0$ there must be some $s\in C$ with nonzero coefficient in $\partial z$ such that $y$ has nonzero coefficient in $\partial s$. Since $C$ is reduced, $s$ must have $(i,j)$ filtration greater than or equal to $x$. We can perform a filtered change of basis $s\mapsto s+x$, such that in the resulting complex, the number of incoming arrows to $x$ is decreased by $1.$ Repeat this process until there is no incoming arrow to $x.$ We infer that
 the dimension of  $H_*(\{i=0\})$ is at least $2$, since in the complex we obtained after the change of basis, both $x$ and $y$ survive into the homology. This contradicts the fact that  $H_*(\{i=0\})\cong \HFa(S^3)$ and $\operatorname{dim}H_*(\HFa(S^3))=1$.
 
 Next, consider the case for any positive integer $n$. The preceding arguments imply that there is no internal differential in $A_g(K)$ that points to the filtration $(\II,\JJ)=(0,\tfrac{2g+n-1}{2n})$ while preserving the $\JJ$ filtration.  The result follows from the fact that $A_g(K)$ is a quotient complex in the truncated mapping cone \eqref{eq:spincgmappingcone}, and the only mapping cone differential $v^\infty_g$ restricted to $\{(\II,\JJ)=(0,\frac{2g+n-1}{2n})\}$ strictly decreases the $\JJ$ filtration. 
\end{proof}

\subsection{Legendrian surgeries that are negative smooth surgeries} 
We will prove a version of Theorem \ref{thm: LOSS determines contact for admissible Legendrian} when the smooth surgery coefficient is negative. %We will use the convention that the contact structures coming from contact negative rational surgeries are all negative stabilizations. 
\begin{theorem} \label{thm: leq -2g-1 surgery}
    Let $K$ be a knot in $S^3$, and $L$ be a Legendrian representative of $K$ with $\tb(L)=-n+1$, $\rot (L)=-n-2g+2$ in some contact $S^3$. Let the contact manifold $(S^3_{-n}(K),\xi_{L_{}})$ be obtained from the contact $(-1)$-surgery (Legendrian surgery) on $L$ and let $L'$ be the Legendrian push-off of $L$ in $(S^3_{-n}(K),\xi_{L})$. Then, for any positive integer $n$, the map $f$ in Lemma \ref{lem: f maps rational LOSS to Contact invariant} that maps $\mathfrak{L}(L')$ to $c(\xi_L)$ is injective. 
\end{theorem}

\begin{proof}
    % What we really need to do here is some checking and calculation about the grading and $\Spin^c$ structures then use Proposition \ref{Prop: inject at top grading} and Lemma \ref{lem: f maps rational LOSS to Contact invariant}. To make it more clear we will make and prove several Claims. 
We start by making two claims:

    \textbf{Claim 1:} The contact invariant $c(\xi_{L}) $ is in $ \HFa(S^3_{n}(-K),1-g)$ where $-K$ is the mirror of $K$.
    
    \textbf{Claim 2:} The LOSS invariant $\mathfrak{L}(L') $ is in $  \HFKm(S^3_n(-K),\mu,g)$ and the Alexander grading of $\mathfrak{L}(L')$ is at the top grading which is $\frac{2g+n-1}{2n}$. 

The theorem then follows directly from Proposition \ref{Prop: inject at top grading} together with these two claims, so it remains only to establish the claims.

% It is clear that once we show both claims the Proposition \ref{Prop: inject at top grading} implies the theorem, so we just need to proof these two claims.

\textbf{Proof of Claim 1:} 

Recall that  the identification between $\spincs \in\Spin^c(S^3_n(K)) $ and $i\in \Z/n\Z$ is given by \[\langle c_1(\underline\spincs),[\Tilde{F}] \rangle  - n \equiv 2i  \quad (\operatorname{mod} 2n ) \] where $\underline\spincs$ is the extension of  $\spincs$ to the cobordism. %When we view the $(-k)$ contact surgery on $L$ as $(-1)$-contact surgery on $S_{-k+1}(L)$, the $k-1$ negative stabilization of $L$, 
Since $\rot(L)=-n-2g+2$ we have 
$$\langle c_1(\underline\spincs),[\Tilde{F}] \rangle-n=-2n-2g+2=-2g+2 \ (mod \ 2n),$$ we have $i=-g+1$ which proves the claim.

\textbf{Proof of Claim 2:}
By Claim 1, the contact invariant lies in the Floer homology group associated to the $\Spin^c$ structure $1-g$. According to  Lemma \ref{lem: f maps rational LOSS to Contact invariant} and  \eqref{eq: g to 1-g map}, we have that the LOSS invariant is in the $\Spin^c$ structure $g$. 

Moreover,
 by \cite[Proposition 2]{DGh}, $L'$ is smoothly isotopic to the dual knot of $L$ in $S^3_{-n}(K)$, hence $\mathfrak{L}(L')\in \HFKm(S^3_n(-K),\mu,g)$, and we left the Alexander grading to find. To determine the Alexander grading, we need to find $\tbq(L')$ and $\rotq(L')$ and then use Theorem \ref{thm: A grading of LOSS}. 

  Those invariants can be calculated using \cite[Lemma 4.1]{DingLiWuContact+1onrationalsphere}. We have 

  \[\tbq(L')=(-n+1)+\frac{-(-n+1)^2}{-n}=\frac{1-n}{n}\]

  and \[\rotq(L')= (-n-2g+2)-\frac{(-n-2g+2)(-n+1)}{-n}=\frac{-n-2g+2}{n}.
  \]
Thus, the Alexander grading \[A(\mathfrak{L}(L'))=\frac{\tbq(L') -\rotq(L')+1}{2}=\frac{2g+n-1}{2n},\] which finishs the proof of the claim 2.

\end{proof}

%\begin{figure}[htb!]\centering\begin{tikzpicture}\begin{scope}[thin, black!0!white]         \draw  (-5, 0) -- (5, 0);  \end{scope}
%    \node at (-4,0){\includegraphics[scale=0.4]{Picture/thm41}};
%    \node at (0,1){\small $(-1)$-contact};
 %   \node at (0,0.5){\small surgery on $L_i^{-k}$};
  %  \node at (6,1){\small $-1$};
  %    \node at (-3.95,1.3){\small $L_i^{-k}$};
   %       \node at (4.05,1.3){\small $L_i^{-k}$};
   %     \draw  [very thick, -stealth]  (-1, 0) -- (1, 0);
    %     \node at (-1.75,-1)[text=blue]{\small $P_i$};
    %      \node at (6.1,-1)[text=blue]{\small $P'_i$};
     %  \node at (4,0){\includegraphics[scale=0.4]{Picture/thm41}};
%\end{tikzpicture}
 %   \caption{ $(-1)$-contact surgery on $L_i^{-k}$ is the same as $(-k-1)$-contact surgery on $L_i$.  }
  %  \label{fig:thm4_1}
%\end{figure} 

\section{Proof of Theorem \ref{thm: LOSS determines contact for admissible Legendrian}}\label{sec: proof of LOSS determines contact for admissible Legendrian}

We first recall the Ding-Geiges-Stipsicz algorithm \cite
{DGS} for representing positive contact surgeries. 

\begin{theorem}[\cite
{DGS} DGS algorithm for $r>0$] {\label{DGS positive contact srugery}}
Given a Legendrian knot $L$ in $(Y,\xi)$. Let $0<x/y=r\in \mathbb{Q}$ be a contact surgery coefficient. Let $e\in \mathbb{Z}$ be the minimal positive integer such that $\frac{x}{y-ex}<0$, with the continued fraction \begin{equation}
    \frac{x}{y-ex}=[a_1,a_2,\dots,a_\l]=a_1-\cfrac{1}{a_2-\cfrac{1}{\dots-\cfrac{1}{a_\l}}}
\end{equation}
where each $a_j\leq -2$. Then any contact $(x/y)$-surgery on $L$ can be described as contact surgery along a link $(L_0^1 \cup L_0^2 \cup \dots \cup L_0^e)\cup L_1 \cup\dots\cup L_{\ell}$, where 
\begin{itemize}
     \item $L_0^1,\dots,L_0^e$ are Legendrian push-offs of $L$.
    \item $L_1$ is obtained from a Legendrian push-off of $L_0^e$ by stabilizing $|a_1+1|$ times.
    \item $L_j$ is obtained from a Legendrian push-off of $L_{j-1}$ by stabilizing $|a_j+2|$ times, for $j\geq 2$.
    \item The contact surgery coefficients are $+1$ on each $L_0^i$ and $-1$ on each $L_j$.
\end{itemize}
    \end{theorem}

Moreover, when we are talking about the contact surgery in general, we always assume the choice of stabilization being negative. We also have the lemma here that helps us understand negative stabilizations.

\begin{lemma} \label{lem: n+1/n surgery and n negative 
stabilization}\cite[Lemma 5.4]{WZnegativecontactsurgery}
    The two Legendrian arcs $e_1$ and $e_2$ depicted in Figure \ref{New L} are Legendrian isotopic. 
\end{lemma}
\begin{figure}[htb!]
\centering
\begin{tikzpicture}
\begin{scope}[thin, black!0!white]
          \draw  (-5, 0) -- (5, 0);
      \end{scope}
    \node at (-3,0){\includegraphics[scale=0.6]{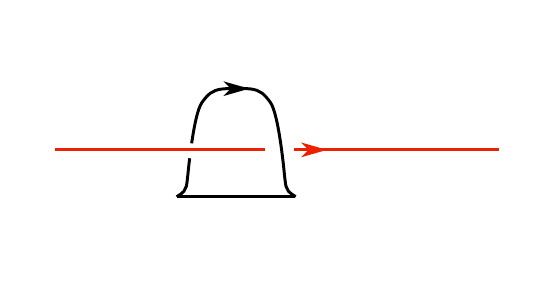}};
     \draw [decorate,decoration={brace,amplitude=4pt},xshift=0.5cm,yshift=0pt]
      (2.8,0.5) -- (2.8,-0.7) node [midway,right,xshift=0.1cm] {\small $n$};
      \node at (-3.5,1.1){\small $+\frac{n+1}{n}$};
        \node at (0,0){ $\sim$};
         \node at (-1.5,.5)[text=red]{ $e_1$};
          \node at (1.8,0.3)[text=red]{ $e_2$};
       \node at (3,0){\includegraphics[scale=0.7]{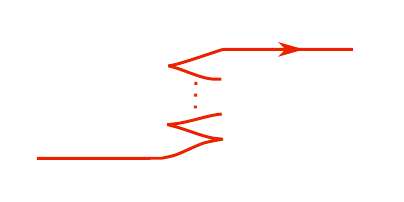}};
\end{tikzpicture}
    \caption{$n$ in the second diagram indicates $n$ zigzags (i.e. $n$ negative-stabilizations). We call the Legendrian knot, along which the $\frac{n+1}{n}$-surgery is performed in the first diagram, the standard Legendrian meridian of the knot which $e_1$ belongs to.}
    \label{New L}
\end{figure}
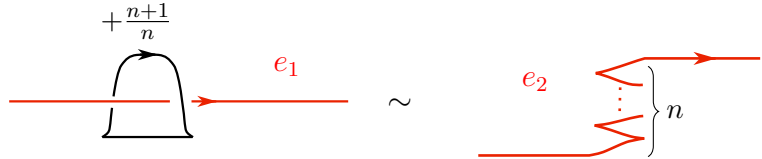 

Now we rewrite the Theorem \ref{thm: LOSS determines contact for admissible Legendrian} as follows. 

\begin{theorem}
 Let $L_1$ and $L_2$ be two Legendrian knots with the same knot type $K$ and $\tb(L_1)=\tb(L_2)$. Let $(Y_i,\xi_i)$ denote the Legendrian surgery on $L_i$, respectively.  
 \begin{enumerate}
     \item If $\mathfrak{L}(L_1)\neq \mathfrak{L}(L_2)$ then $c(\xi_1)\neq c(\xi_2)$.
     \item If $\mathfrak{L}(L_1)\neq 0$ then $c(\xi_1)\neq0$.
 \end{enumerate}   
\end{theorem}

\begin{proof}
    %The case for $tb(L_i)\leq 2g-1$ follows directly from the injectivity Theorem \ref{thm: leq -2g-1 surgery} and Corollary \ref{cor: Legen surgery on distinct loss give distinct loss}.  To deal with the case where $\tb$ is not small enough we apply the trick of interpreting stabilization by contact surgery, more specifically we will frequently use the lemma above. 

%Back to the situation when $2g-1 \leq \tb(L_i)$. We will divide it into $3$ cases and prove the first statement in the Theorem: 

%We again let $L_i'$ to be the Legendrian push-off of $L_i$ in $(Y_i,\xi_i)$. 
We will divide it into $3$ cases and prove the first statement first. 

\begin{enumerate}
    \item[Case 1.] $\tb(L_i)\leq 0$.
    
    This situation follows directly from Corollary \ref{cor: Legen surgery on distinct loss give distinct loss} and Theorem \ref{thm: leq -2g-1 surgery}. %$2g \leq  \tb(L_i)\leq 0$. 
    
%We first find some positive integer $n$ such that, $L_i^{-n}$, the $n$ negatively stabilized $L_i$ has the property that $L_i^{-n}$ are admissible, $\tb(L_1^{-n})=\tb(L_1^{-n})\leq 2g-1$. Moreover, Theorem \ref{LOSS invariant under negative stabilization} and the condition $\mathfrak{L}(L_1)\neq \mathfrak{L}(L_2)$ implies $\mathfrak{L}(L_1^{-n})\neq \mathfrak{L}(L_2^{-n})$. Let $(X_i,\delta_i)$ be the contact $3$-manifold obtained by taking Legendrian surgery on $L_i^{-n}$ respectively, then the $\tb \leq 2g-1$ case tells $c(\delta_1)\neq c(\delta_2)$. 

%Now the Lemma \ref{lem: n+1/n surgery and n negative 
%stabilization} above says that $(X_i,\delta_i)$ is obtained by taking $+\frac{n+1}{n}$-contact surgery (which is equivalent to a sequence of $+2$-contact surgeries) on the standard Legendrian meridian of $L_i$. Note that each steps of performing $+2$-contact surgery we are starting and ending at some rational homology sphere, thus the naturality of contact invariant under positive contact surgery (Theorem \ref{thm: same Spin^c for Legendrian with classical invariants}) implies that the $HF$ maps $c(\delta_i)$ to $c(\xi_i)$. Hence, we have $c(\xi_1)\neq c(\xi_2)$. See Figure \ref{fig:n+1/n contact surgery}.

    \item[Case 2.] $2<\tb(L_i)$ 

Let us first assume when $\tb(L_i)=2$ the statement is true and prove the case for $2<\tb(L_i)$, and in the third case we will prove the situation when $\tb(L_i)=2$.  

Since $L_i$ is admissible, we can first find some positive integer $n$ such that, $L_i^{-n}$, the $n$ negatively stabilized $L_i$  are admissible and $\tb(L_1^{-n})=\tb(L_2^{-n})=2$. Moreover, Theorem \ref{LOSS invariant under negative stabilization} and the condition $\mathfrak{L}(L_1)\neq \mathfrak{L}(L_2)$ implies $\mathfrak{L}(L_1^{-n})\neq \mathfrak{L}(L_2^{-n})$. Let $(X_i,\delta_i)$ be the contact $3$-manifold obtained by taking the Legendrian surgery on $L_i^{-n}$ respectively, then the assumption of $\tb =2 $ case tells us $c(\delta_1)\neq c(\delta_2)$. 

Now the Lemma \ref{lem: n+1/n surgery and n negative 
stabilization} above says that $(X_i,\delta_i)$ is obtained by taking contact $(+\frac{n+1}{n})$-surgery (which is equivalent to a sequence of $(+2)$-contact surgeries) on the standard Legendrian meridian of $L_i$ in $(Y_i,\xi_i)$. Note that, for each step of performing a contact $(+2)$-surgery, we are start and end at some rational homology spheres, and thus the naturality of contact invariant under positive contact surgery (Theorem \ref{thm: same Spin^c for Legendrian with classical invariants}) implies that the $HF$ maps $c(\xi_i)$ to $c(\delta_i)$ for $i=1,2$ respectively. Hence, we have $c(\xi_1)\neq c(\xi_2)$. See Figure \ref{fig:n+1/n contact surgery}
    \item[Case 3.] $\tb(L_i)=2$.
    
This is the most complicated case. When we try to follow the same argument as the second case (going from $\tb=2$ to $\tb=0$), we will end at a non-rational homology sphere after we perform the first $(+2)$-contact surgery in the sequence. Hence, we can not use the similar naturality argument as Case 2.  However, there are some tricks to get around it.  

Let $(Y^3_i,\xi^3_i)$ be the contact manifold obtained by taking contact $(+3)$-surgery along the standard Legendrian meridian of $L_i$ in $(Y_i,\xi_i)$. 

We first make a claim $(a)$:  the contact invariant $c(\xi^3_1) \neq c(\xi^3_2)$. If the claim $(a)$ is true, then the naturality condition is satisfied, so the same argument as in case $(1)$ will imply $c(\xi_1)\neq c(\xi_2)$.  

\emph{Proof of claim $(a)$:} We denote the standard Legendrian meridian of $L_i$ as $\prescript{}{i}U$. By applying the DGS algorithm \ref{DGS positive contact srugery}, we can represent contact $(+3)$-surgery on $\prescript{}{i}U$ as contact surgery along $\prescript{}{i}U_0\cup \prescript{}{i}U_1\cup \prescript{}{i}U_2$ where $\prescript{}{i}U_0$ is the Legendrian push-off of $\prescript{}{i}U$, $\prescript{}{i}U_1$ is the Legendrian push-off of $\prescript{}{i}U_0$ with one negative stabilization, and $\prescript{}{i}U_2$ is the Legendrian push-off of $\prescript{}{i}U_1$, with contact coefficient $(+1)$ on $\prescript{}{i}U_0$ and $(-1)$ on the other two.

 We now make a claim $(b)$: After performing a contact $(+2)$-surgery on the standard meridian of $\prescript{}{i}U_1$ in $(Y_i^3,\xi_i^3)$, the result contact manifolds $(Y_i^{3,2},\xi_i^{3,2})$ have different contact invariants. Note that if the claim $(b)$ is true then again the naturality condition is satisfied and the claim $(a)$ follows.

\emph{Proof of claim $(b)$:} First observe that by Lemma \ref{lem: n+1/n surgery and n negative 
stabilization}, contact $(-1)$-surgery on $\prescript{}{i}U_1$ followed by contact contact $(+2)$-surgery on the standard meridian of $\prescript{}{i}U_1$ is equivalent to contact $(-1)$-surgery on $\prescript{}{i}U_1^{-1}$, the negative stabilization of $\prescript{}{i}U_1$. 

As a consequence, $(Y_i^{3,2},\xi_i^{3,2})$ can be viewed as taking contact $(+1)$-surgery on $\prescript{}{i}U_0$, contact $(-1)$-surgery on $\prescript{}{i}U_1^{-1}$ (which is also the push-off of $\prescript{}{i}U_0$ with \emph{two} negative stabilizations), and contact $(-1)$-surgery on $\prescript{}{i}U_2$. According to the DGS algorithm \ref{DGS positive contact srugery}, the contact $(+1)$-surgery on $\prescript{}{i}U_0$ followed by the contact $(-1)$-surgery on  $\prescript{}{i}U_1^{-1}$ is equivalent to contact $(+\frac{3}{2})$-surgery on $\prescript{}{i}U$. Thus, $(Y_i^{3,2},\xi_i^{3,2})$ can also be obtained from $(Y_i,\xi_i)$ (recall this is the Legendrian surgery on $L_i$) by first taking contact $(+\frac{3}{2})$-surgery on $\prescript{}{i}U_0$ followed by the Legendrian surgery on $\prescript{}{i}U_2$, which by Lemma \ref{lem: n+1/n surgery and n negative 
stabilization} again is equivalent to performing the Legendrian surgery on $L_i^{-2}$ followed by the Legendrian surgery on $\prescript{}{i}U_2$. 

Now $L_i^{-2}$ satisfies the case $(1)$ condition, so Legendrian surgeries on $L_i^{-2}$ give different contact invariants. Moreover, since a Legendrian surgery preserves the distinction between contact invariants (for example see \cite[Theorem 1.1]{Wantightcontactstructures}), after taking Legendrian surgeries on $\prescript{}{i}U_2$  (previous argument tells us this give us back to $(Y_i^{3,2},\xi_i^{3,2})$) we still have distinct contact invariants. Thus, claim $(b)$ is true which implies claim $(a)$ is true. This concludes the proof of the case when $\tb(L_i)=2$.

%adopt the same notation as the first case where $\tb(L_i)=2$ with $n=2$. Then $\tb(L_i^{-2})=0$ and by the (1) case the Legendrian surgeries $(X_i,\delta_i)=$ are integer homology spheres with $c(\delta_1)\neq c(\delta_2)$. We instead think 
    
   % \item %$2<\tb(L_i)$
    
    %The similar argument as the first case says that we can take a sequence of $+2$-contact surgeries on the standard Legendrian meridians and stop at such that the equivalent number of negative stabilizations have $\tb$ equals to $2$, then use the result of second case, and the rest follows.
\end{enumerate} 

Exactly the same argument will show the second statement in the theorem that $\mathfrak{L}(L_1)\neq 0$ implies $c(\xi_1)\neq 0$ holds as well.

\end{proof}

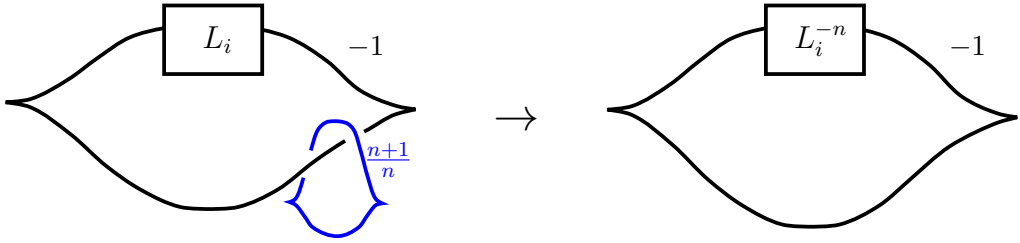
\begin{figure}[htb!]
\centering
\begin{tikzpicture} 
\begin{scope}[thin, black!0!white]
          \draw  (-5, 0) -- (5, 0);
      \end{scope}
    \node at (-4,0){\includegraphics[scale=0.5]{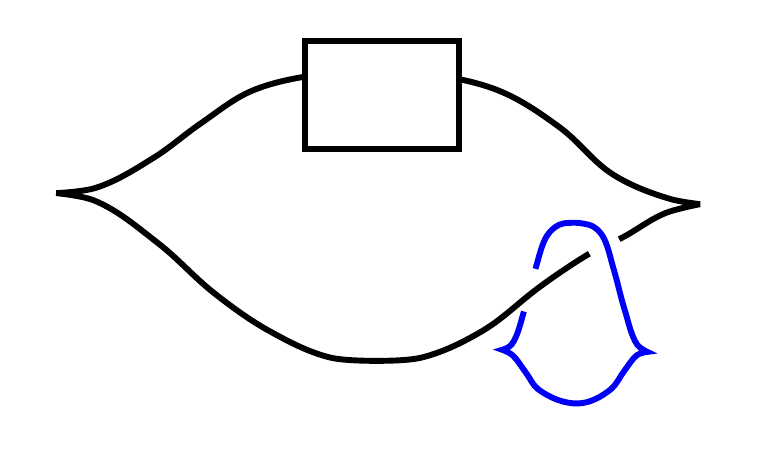}};
     \node at (-2,1){\small $-1$};
      \node at (-1.7,-0.5)[text=blue]{\small $\frac{n+1}{n}$};
    \node at (6,1){\small $-1$};
      \node at (-3.97,1.1){ $L_i$};
          \node at (4.05,1.1){ $L_i^{-n}$};
    \node at (0,0){\Large $\rightarrow$};
       \node at (4,0){\includegraphics[scale=0.5]{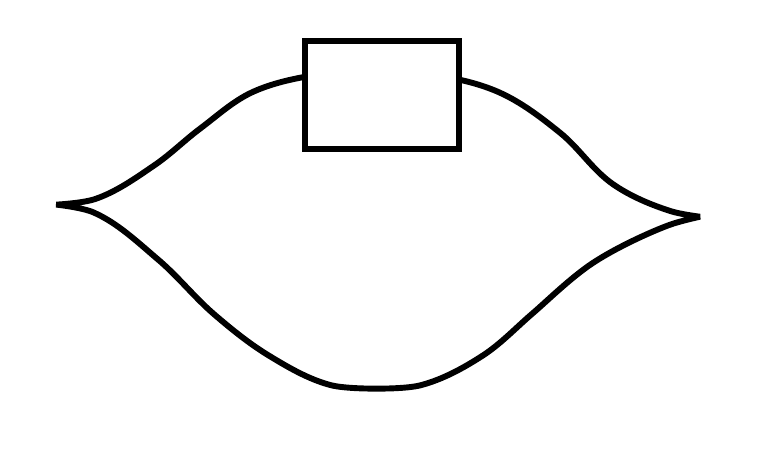}};
\end{tikzpicture}
\caption{The Legendrian surgery on $L_i^{n}$ is obtained by contact $(+\frac{n+1}{n})$-surgery (a sequence of $(+2)$-contact surgeries) on the standard Legendrian meridian of the Legendrian surgery on $L_i$. }
\label{fig:n+1/n contact surgery}
\end{figure}

\section{Construction of admissible Legendrian representatives}\label{sec: construction of admissible Legendrian}
In this section, we use techniques from balanced sutured manifolds and sutured Floer homology to establish Theorem \ref{thm: 2g-1 non-loose knot}. 

\subsection{Balanced sutured manifolds}
Suppose $M$ is a compact oriented $3$-manifold with non-empty boundary. Let $\gamma=(A(\gamma),s(\gamma))$ be a pair where $A(\gamma)\subset\partial M$ is a disjoint union of embedded annuli on $\partial M$ and $s(\gamma)$ is the core of $A(\gamma)$ and is oriented. Since $\partial A(\gamma)$ is parallel to $s(\gamma)$, we orient $\partial A(\gamma)$ according to the given orientation of $s(\gamma)$. We can decompose $R(\gamma) = \overline{\partial M - A(\gamma)}$ into two parts:
\[
R(\gamma) = R_+(\gamma) \cup R_-(\gamma),
\]
where $R_+(\gamma)$ is the part of $R(\gamma)$ such that the orientation induced by $\partial M$ coincides with the orientation induced by $\partial R(\gamma) = \partial A(\gamma)$, and $R_-(\gamma) = R(\gamma) - R_+(\gamma)$.

\begin{remark}
    Our definition of the suture $\gamma$ is more restrictive than the original definition in \cite{gabai1983foliations}, in the sense that we exclude the possibility that $\gamma$ can contain some toroidal components, as such components cannot exist for balanced sutured manifolds. Since $A(\gamma)$ is simply a tubular neighborhood of $s(\gamma)$, on some occasions we will simply use $s(\gamma)$ in place of $\gamma$.
\end{remark}

\begin{definition}[{\cite[Definition 2.2]{juhasz2006holomorphic}}]
Suppose $(M,\gamma)$ is a pair as described above. If $M$ is connected, then we call $(M,\gamma)$ a {\bf balanced sutured manifold} if the following hold.
\begin{itemize}
	\item For every component $V\subset\partial M$, $V\cap A(\gamma)\neq \emptyset$.
	\item We have $\chi(R_+(\gamma)) = \chi(R_-(\gamma)).$
\end{itemize}
If $M$ is disconnected, then we call $(M,\gamma)$ a balanced sutured manifold if for any component $C\subset M$, $(C,C\cap\gamma)$ is a balanced sutured manifold in the above sense.
\end{definition}

\begin{definition}
   Suppose $(M,\gamma)$ is a balanced sutured manifold. We say that a co-oriented contact structure $\xi$ is compatible with $(M,\gamma)$ if $\partial M$ is convex and $\gamma$ is (isotopic to) the dividing curve on $\partial M$.  
\end{definition}

\begin{theorem}[{\cite[Theorem 0.1 and Corollary 4.3]{honda2009contact}}]
Suppose $\xi$ is a contact structure compatible with $(M,\gamma)$. Then there is a well-defined element
\[
c(\xi) \in SFH(-M,-\gamma).
\]
Furthermore, if $\xi$ is overtwisted, then $c(\xi) = 0$.
\end{theorem}

\subsection{Contact structures and contact elements}
We first establish a lemma about the existence of non-vanishing contact invariant on a taut balanced sutured manifold. 
\begin{lemma}\label{lem: contact str with non-zero contact element}
	Suppose $(M,\gamma)$ is a taut balanced sutured manifold, then there exists a tight contact structure $\xi$ on $(M,\gamma)$ such that
\[
c(\xi)\neq 0 \in SFH(-M,-\gamma).
\]
\end{lemma}
\begin{proof}
	By \cite[Theorem 4.2]{gabai1983foliations}, $(M,\gamma)$ admits a sutured manifold hierarchy
	\begin{equation}\label{eq: hierarchy}
		(M,\gamma) \stackrel{S_1}{\leadsto} (M_1,\gamma_1) \cdots \stackrel{S_n}{\leadsto} (M_n,\gamma_n)
	\end{equation}
	where $S_i$ has no closed components and is a well-groomed surface, in the sense of \cite[Definition 2.6]{juhasz2010polytope}, for all $i$, and $(M_n,\gamma_n)$ is the disjoint union of taut sutured $3$-balls. Let $B^3$ be a $3$-ball and $\delta$ be a suture on $\partial B^3$. It is well-known that $(B^3,\delta)$ is taut if and only if $\delta$ is connected and there is a unique tight contact structure $\xi_{st}$ on $(B^3,\delta)$ such that
	\[
	c(\xi_{st})\neq 0\in SFH(-B^3,-\delta)\cong\mathbb{F}.
	\]
	Then we can glue back the desired tight contact structure through the hierarchy by {\cite[Theorem 6.2]{honda2009contact}}.
\end{proof}

Now we are ready to prove Theorem \ref{thm: 2g-1 non-loose knot}.
\begin{proof}[Proof of Theorem \ref{thm: 2g-1 non-loose knot}]
	We adopt the same notation as in \cite[Section 3]{LW2024}. In particular, we use $\Gamma_k$ to denote the suture on $\partial (S^3\backslash N(K))$ consisting of two curves of slope $-k$. Let $S$ be a minimal-genus Seifert surface of $K$. As in the proof of \cite[Theorem 3.4]{LW2024}, there is a sutured manifold decomposition
	\[
	(S^3\backslash N(K), \Gamma_k)\stackrel{S}{\leadsto} (S^3\backslash N(S),\gamma^1)
	\]
	for any $k\geq 1$, and
	\[
	(S^3\backslash N(K), \Gamma_k)\stackrel{S}{\leadsto} (S^3\backslash N(S),\gamma^3)
	\]
	for $k=0$. Here $\gamma^1$ is a single copy of $\partial S$ and $\gamma^3$ consists of three parallel copies of $\partial S$. A further product annular decomposition yields
	\[
	(S^3\backslash N(S),\gamma^3)\stackrel{A}{\leadsto} (S^3\backslash N(S),\gamma^1).
	\]
	Now, applying Lemma \ref{lem: contact str with non-zero contact element} we obtain a contact structure $\xi_S$ on $(S^3\backslash N(S),\gamma^1)$, and the proof of \cite[Theorem 3.4]{LW2024} enables us to glue $\xi_S$ to a family of tight contact structures $\xi_k$ on $S^3\backslash N(K)$ for $k\geq 0$ such that the following hold.
	\begin{itemize}
		\item [(c1)] For any $k\geq 0$, we have 
		\[c(\xi_k)\neq 0 \in SFH(-S^3\backslash N(K),-\Gamma_k,g(K)+\frac{k-1}{2}).
		\]
		\item [(c2)] For any $k\geq 0$, $\xi_{k+1}$ is obtained from $\xi_k$ by attaching a negatively signed basic slice as in \cite[Section 2.1.3]{etnyre2017sutured}. In particular, if $\psi^{k}_{-,k+1}$ is the contact gluing map associated to the corresponding negatively signed basic slice (cf. \cite[Section 1.1]{etnyre2017sutured}), then
		\[
		\psi^k_{-,k+1}(c(\xi_k)) = c(\xi_{k+1}).
		\]
	\end{itemize}
	
The proof of \cite[Lemma 2.20]{LW2024} indicates that the map $\psi^{k}_{-, k+1}$ is an isomorphism restricted to grading $g(K)+\frac{k-1}{2}$ when $k\geq 1$. Note this map is the map used in \cite[Section 1.1]{etnyre2017sutured} to construct the direct system whose direct limit recovers $HFK^-(S^3,-K)$ (cf. \cite[Equation (2.6)]{LW2024}). Thus, the collection $\{c(\xi_k)\}_{k\geq 0}$ gives rise to a non-zero element in $HFK^-(S^3,-K,g(K))$. Here the grading of the limit element can be computed explicitly by the grading of $c(\xi_k)$ and the explicit grading shifting formula in the definition of the direct system. 

Next, for each $k\geq 0$, we can glue a standard tight contact solid torus to $(S^3\backslash N(K),\Gamma_k)$ to form a contact structure $\xi$ on $S^3$, and let $L^{-k}$ be the core of the contact solid torus, which is a Legendrian representation of $K$ with respect to the contact structure $\xi$. A priori, $\xi$ depends on the index $k$, but as in \cite[Section 1.2]{etnyre2017sutured}, the negative stabilization of $L^{-k}$ corresponding to attaching a negatively signed basic slice, which coincides with the construction of $L^{-k-1}$, and thus $\xi$ is independent of $k$. When we glue back the standard solid torus to $(S^3\backslash N(K),\Gamma_k)$, we see that ${\rm tb}(L^{-k}) = -k$. The rotation number comes from the fact that $c(\xi_k)$ is in the grading $g(K)+\frac{k-1}{2}$.
Furthermore, by \cite[Theorem 1.5]{etnyre2017sutured},
\[
\mathfrak{L}(L^{-k})\neq 0 \in HFK^-(-S^3,K,g(K)),
\]
where $\mathfrak{L}$ denotes the LOSS invariant.

In particular, if the contact structure $\xi$ we construct above is tight then $\tb(L)-\rot(L)=2g-1$ implies $\tau(K)=g(K)$, Thus, when $\tau(K)<g(K)$ the contact structure we construct has to be overtwisted and $L^{-k}$ is non-loose since $\mathfrak{L}(L^{-k})\neq 0$.

\section{proof of Theorem \ref{thm: negative surgery on knot is tight}} \label{sec: proof of negative surgery on knot is tight}
We are ready to prove Theorem \ref{thm: negative surgery on knot is tight}.
\begin{proof}
    For any positive rational number $r$, we perform contact $(-r)$-surgery (with the stabilization choices being all negative) on the admissible Legendrian representative $L^0$ we constructed in Theorem \ref{thm: 2g-1 non-loose knot}. Recall that $\tb(L^0)=0$ and $\rot(L^0)=-2g+1$, thus smoothly we are also doing $-r$ surgery. Now the Ding-Geiges-Stipsicz \cite{DGS} algorithm tells us that contact $(-r)$-surgery is equivalent to a sequence of Legendrian surgeries on push-offs of $L^0$ with some negative stabilizations. We start with the first Legendrian surgery on some negative stabilizations of $L^0$, and by Corollary \ref{thm: Legendrian surgery with non-vanishing contact invariant}, this Legendrian surgery has non-vanishing contact invariant. Moreover, since a Legendrian surgery preserves the non-vanishing of contact invariant, the result follows.
\end{proof}

%It remains to show that $\xi$ is overtwisted on $S^3$. Note that $S^3$ has a unique tight contact structure whose contact element is non-zero. As a result, it suffices to show that $c(\xi) = 0$. Recall that $\xi$ is obtained from $\xi_k$ by filling in a standard tight contact solid torus. This gives rise to a map
%\[
%\pi^k : SFH(-S^3\backslash N(K),-\Gamma_k)\to \widehat{HF}(-S^3),
%\]
%such that $\pi_k(c(\xi_k)) = c(\xi)$ and $\pi^k$ intertwine with $\psi^k_{-,k+1}$ in the sense of \cite[Equation 2.8]{LW2024}. By \cite[Theorem 1.3]{etnyre2017sutured} (also cf. \cite[Theorem 2.19]{LW2024}), we know that the collection of maps $\{\pi^k\}_{k\geq 0}$ induces a map
%\[
%\pi_*: {HFK}^-(-S^3,K)\to \widehat{HF}(-S^3)
%\]

%such that $c(\xi) = \pi_*(\mathfrak{L}(L^{-k}))$. Since $\tau(K) < g(K)$ we know that (as one definition of $\tau$) $\pi_*$ restricts to the zero map on the top grading of $HFK^-(-S^3,K)$, while $\mathfrak{L}(L^{-k})$ sits precisely on the top grading, and hence $c(\xi) = 0$.
\end{proof}

%\begin{proof}[Proof of Theorem \ref{thm: 2g-1 non-loose knot, 2}]
	%The proof is almost identical to that of Theorem \ref{thm: 2g-1 non-loose knot}. The only difference is that we apply Lemma \ref{lem: contact str with non-zero contact element, 2} to obtain two contact structures whose contact elements are linearly independent on $(S^3\backslash N(K),\gamma^1)$ instead of just one contact structure from Lemma \ref{lem: contact str with non-zero contact element}. Note that we have $H_2(S^3\backslash N(S);\mathbb{Z}) = 0$ by \cite[Lemma 5.1]{juhasz2010polytope}.\end{proof}

\bibliographystyle{amsalpha}
\bibliography{bibliography}

\end{document}